\documentclass[11pt]{amsart}
\usepackage{amsmath, amsthm, amscd, amsfonts}
\allowdisplaybreaks 
\usepackage{tikz}
\usepackage{pstricks,pst-node}
\usepackage[all]{xy}

\makeatletter \oddsidemargin.9375in \evensidemargin \oddsidemargin
\newtheorem{theorem}{Theorem}[section]
\newtheorem{lemma}[theorem]{Lemma}
\newtheorem{proposition}[theorem]{Proposition}
\newtheorem{corollary}[theorem]{Corollary}
\theoremstyle{definition}
 \newtheorem{definition}[theorem]{Definition}

\theoremstyle{remark}

\numberwithin{equation}{section}

\begin{document}

\title[ Operators preserving the representation of a semidirect product group]{Operators preserving the representation of a semidirect product group}

\author[M. Mortazavizadeh, R. Raisi Tousi]{M. Mortazavizadeh, R. Raisi Tousi$^{*}$\\
October 13, 2019}



\begin{abstract}
For a locally compact abelian group $\textbf{R}$ with a uniform lattice $\Lambda$ and a group $G$ that acts on $\textbf{R}$ by continuous automorphisms, we study operators commuting with the representation of $G \times \Lambda$ on $L^2(\textbf{R})$. As a consequence, we give a characterization of shift- dilation preserving operators.

\end{abstract} \maketitle

\section{Introduction and Preliminaries}

\noindent
Shift invariant subspaces of $L^2(\Bbb{R}^d )$ are closed subspaces of $L^2(\Bbb{R}^d )$ that are invariant under integer shifts. They play an essential role in many area of mathematical analysis and its applications \cite{BDR, B, RS}. Shift invariant subspaces of $L^2(\Bbb{R}^d )$ are introduced by a range function approach in \cite{Hel}. M. Bownik gives a characterization of these spaces in terms of range functions in \cite{B}. He also studied shift preserving operators on $L^2(\Bbb{R}^d )$, bounded linear operators that commute with integer shifts, and characterized them via range operators. The theory of shift invariant spaces is generalized to the setting of locally compact abelian groups in \cite{Cab, KRr, KRs}. For a locally compact abelian group $\textbf{R}$ with a uniform lattice $\Lambda$, a shift invariant subspace is a closed subspace of $L^2(\textbf{R})$ which is invariant under shifts by elements of $\Lambda$. A bounded linear operator $U$ on $L^2(\textbf{R})$ is said to be shift preserving if $U$ commutes with the shift operators by elements of $\Lambda$. The structure of these operators are studied in \cite{KRsh} in which the authors give a characterization of shift preserving operators in terms of range operators. In \cite{BHP}, the structure of subspaces that are invariant under the action of a discrete locally compact abelian group, as a generalization of shift invariant spaces, is investigated. Recently, the authors in \cite{BCHM} studied the structure of spaces that are invariant under the action of a discrete locally compact group (not necessarily abelian) that takes the form of a semidirect product. Following an idea of \cite{BCHM}, our goal in this paper is to study the operators commuting  with the action of a semidirect product group. More precisely, assume that $\textbf{R}$ is a  locally compact abelian group  and $\Lambda$ is a uniform lattice of $\textbf{R}$, assume also that $G$ is a group that acts on $\textbf{R}$ by continuous automorphisms. In \cite{BCHM}, the authors setted $\Gamma = \Lambda \times  G $ and showed that there is a unitary representation $T_k R_g$ of $\Gamma$ on  $L^2(\textbf{R})$ which is composition of the shift operator and a certain unitary operator. They then studied the structure of  the spaces invariant under $T_k R_g$ in terms of range functions. In this paper we define a $\Gamma$-preserving operator as a bounded linear operator on $L^2(\textbf{R})$ that commute with $T_k R_g$, and then we study the structure of the such operators. First of all, we give an equivalent condition on a bounded linear operator to be $\Gamma$-preserving. Then, we prove our main result that states that $\Gamma$-preserving are exactly shift preserving operators with an extra condition on the range operator.  As a application of our main result, we give a characterization of shift dilation invariant spaces (preserving operators) in terms of range functions (range operators). This paper is organized as follows. The rest of this section contains the preliminaries related to shift invariant spaces and $\Gamma$-invariant spaces which are studied in \cite{KRs, KRsh} and \cite{BCHM}, respectively. In Section 2, we state our main result in Theorem \ref{Main} which determines the structure of $\Gamma$-preserving operators. Indeed, we show that a bounded linear $U$ is $\Gamma$-preserving if and only if it is shift preserving and its range operator satisfies an extra condition. Also as a consequence of Theorem \ref{Main} we characterize shift- dilation invariant spaces and shift- dilation preserving operators. To this end we show that the composition of translation and dilation can be a representation of the semidirect product $Aut(\textbf{R}) \times \Lambda$ on $L^2(\textbf{R})$. Applying Theorem \ref{Main}, we get  characterizations on shift- dilation invariant spaces and shift- dilation preserving operators.

Let $\textbf{R}$ be a second countable locally compact abelian group with a uniform lattice $\Lambda$. A closed subspace $\mathcal{V}$ of $L^2(\textbf{R})$ is said to be shift invariant if $T_k \mathcal{V} \subseteq \mathcal{V}$ for all $k \in \Lambda$.  A range function is a mapping $J: \Omega \longrightarrow \lbrace \textrm{  closed subspaces  of   $l^2 (\Lambda ^{\perp})$ } \rbrace$. A range function $J$ is measurable if the mapping $\omega \mapsto \langle P_{J}(\omega)(a) , b \rangle$ is measurable for all $a ,b \in l^2 (\Lambda ^{\perp}) $, where $P_{J}(\omega)$ for the orthogonal projections of $l^2 (\Lambda ^{\perp})$ onto $J(\omega)$. In \cite{KRsh}, the authors characterized shift invariant spaces in terms of  range functions as follows.
\begin{proposition}
Let $\mathcal{V} \subseteq L^{2}(\textbf{R})$ be a closed subspace and $\mathcal{T}$ be the fiberization map. Then the following are equivalent.\\
(1)  $\mathcal{V}$ is a shift invariant  space. \\
(2) There exists a measurable range function $J: \Omega \longrightarrow \lbrace  closed \ subspaces \  of \   l^2 (\Lambda ^{\perp}) \rbrace $ such that 
\begin{equation*}
\mathcal{V}=\lbrace f \in L^{2}(\textbf{R}) : \mathcal{T}(f)(\omega)\in J(\omega) \text{,} \ \  \text{for a.e. }  \omega \in \Omega \rbrace.
\end{equation*}
Identifying range functions which are equivalent almost everywhere, the correspondence between shift invariant spaces and measurable range functions is one to one and onto. 
\end{proposition}
Let $\mathcal{V}$ be a shift invariant  subspace of $L^2 (\textbf{R})$  with range function $J$.  A range operator on $J$ is a mapping $\mathcal{R}$ from the fundamental domain $\Omega$ of $\widehat{\textbf{R}}/ \Lambda ^{\perp} $ to the set of all bounded linear operators on closed subspaces of $l^2(\Lambda^{\perp})$, so that the domain of $\mathcal{R}(\omega)$ equals $J(\omega)$ for a.e. $\omega \in \Omega$. A range operator $\mathcal{R}$ is called measurable, if the mapping $\omega \mapsto \langle \mathcal{R}(\omega)P_{J}(\omega)(a) , b \rangle$ is measurable for all $a,b \in l^2(\Lambda^{\perp})$, where $P_{J}(\omega)$ is the orthogonal projection of $l^2(\Lambda^{\perp})$ onto $J(\omega)$. A bounded linear operator $U$ on $L^2(\textbf{R})$ is said to be shift preserving associated if $UT_{k} = T_{k}U$, where $T_{k}$ is the shift operator. The next proposition, that is in \cite{KRsh}, gives a characterization of shift preserving operators in terms of the range operators by a range function approach.
\begin{proposition} 
Let $ \mathcal{V} \subseteq L^{2}(\textbf{R})$ be a  shift invariant subspace with range function $J$ and $U: \mathcal{V} \longrightarrow L^2(\textbf{R})$ be a bounded linear operator. Then the following are equivalent. \\
(1) The operator $U$ is shift preserving. \\
(3) There exists a measurable range operator $\mathcal{R}$ on $J$ such that for all $\phi \in \mathcal{V}$,
\begin{equation} \label{r.opg}
\mathcal{T} U\phi(\omega) = \mathcal{R}(\omega)(\mathcal{T}\phi(\omega)) \ \ a.e. \ \omega \in \Omega ,
\end{equation}
where $\mathcal{T}$ is the fiberization map.
\end{proposition}
Now let $G$ be a group that acts on $\textbf{R}$ by continuous automorphisms. More precisely, suppose that there exists a continuous action $G \times \textbf{R} \longrightarrow \textbf{R}, \ (g,x)\mapsto gx$ such that the mapping $\textbf{R} \longrightarrow \textbf{R}, \ x\mapsto gx$ is an automorphism. Let $\widehat{\textbf{R}}$ be the dual group of $\textbf{R}$ and $\Lambda ^{\perp}$ be  the annihilator of $\Lambda$ in $\widehat{\textbf{R}}$. The action of $G$ on $\textbf{R}$ induces an action of $G$ on $\widehat{\textbf{R}}$ by duality $\langle g^* , \xi \rangle =\langle \xi, gx \rangle$, where $\xi \in \widehat{\textbf{R}}$. This action satisfies $g_1 ^ * g_2 ^* = (g_2 g_1)^*$. Assume that the action of $G$ on $\textbf{R}$ preserves $\Lambda$, i.e. $g\Lambda = \Lambda$. Equivalently, assume that $g^* \Lambda^{\perp}  = \Lambda^{\perp}$. For a fundamental domain $\Omega$ of $\widehat{\textbf{R}} / \Lambda^{\perp}$, the action of $G$ on $\Omega$ is denoted for simplicity by the same notation. 

Using the fact that the action of $G$ preserves $\Lambda$, one can define  the semidirect product $\Gamma = \Lambda \times  G =\lbrace (k,g) : k \in \Lambda , g \in G \rbrace$ with
\begin{equation*}
(k,g) . (k',g') = (k + gk' , gg').
\end{equation*}
Now the action of $\Gamma$ on $\textbf{R}$ given by 
\begin{equation*}
\gamma x = gx+k, \ \ \ \   \  \gamma= (k, g) \in \Gamma , x \in \textbf{R} 
\end{equation*}
will provide the symmetry with respect to which we study invariance. 

For a Hilbert space $\mathcal{H}$, we denote by $\mathcal{U}(\mathcal{H})$ the unitary operators on $\mathcal{H}$. We work with the following representations. 
\begin{align*}
T : \Lambda \longrightarrow \mathcal{U}(L^2(\textbf{R})), \ \  T_k f(x)= f(x-k) \\
R : G \longrightarrow \mathcal{U}(L^2(\textbf{R})), \ \ R_g f(x)= f(g^{-1}x),
\end{align*}
for $k \in \Lambda$, $g \in G$ and $f \in L^2(\textbf{R})$. It can be shown that $(k,g) \mapsto T_k R_g $ defines a unitary representation of $\Gamma$ on $L^2(\textbf{R})$. The following lemma which is \cite[Lemma 2.2]{BCHM} describes the Fourier image of $T_k$ and $R_g$.
\begin{lemma}
For all $f \in L^2(\textbf{R})$ and $(k,g) \in \Gamma$ 
\begin{equation*}
\widehat{T_k f}(\xi) = \langle \xi , k \rangle \widehat{f}(\xi) \ , \   \widehat{R_g f}(\xi)= f(g^* \xi).
\end{equation*}
\end{lemma}
In \cite{KRsh} it is shown that there exists an isometric isomorphism $\mathcal{T}$, known as "fiberization map", between Hilbert spaces $L^2(\textbf{R})$ and $L^2(\Omega , l^2(\Lambda^{\perp}))$ as follows
\begin{equation*}
\mathcal{T}f(\omega) = \lbrace \widehat{f}(\omega + k)  \rbrace_{k \in \Lambda^{\perp}}.
\end{equation*}
In \cite{BCHM}, the authors defined a representation of $G$ on $l^{2}(\Lambda^{\perp})$ by
\begin{equation} \label{r}
(r_g (a))(s) = a(g^* s),
\end{equation}
 for $a \in l^{2}(\Lambda^{\perp})$ and $s \in   \Lambda^{\perp}$, and proved the following property 
\begin{equation*}
\mathcal{T}(T_k R_g f)(\omega) = \langle \omega,k \rangle r_g \mathcal{T}f(g^* \omega).
\end{equation*}
In \cite{}, also a representation $\Pi$ of $G$ on $L^2(\Omega , l^2(\Lambda^{\perp}))$ is defined as
\begin{equation} \label{P}
 \Pi(g) = \mathcal{T}R_g \mathcal{T}^{-1}
\end{equation}
and the following relation between $\Pi(g)$ and $r_g$ is proved.
\begin{equation} \label{Prg}
\Pi(g) F(\omega)= r_g F(g^* \omega), \ \  F \in L^2(\Omega , l^2(\Lambda^{\perp})) .
\end{equation}

Now we introduce our invariant spaces which is recently studied in \cite{BCHM}. A closed subspace $V$ of  $ L^{2}(\textbf{R})$ is called $\Gamma$-invariant if $T_k R_g V \subseteq V$ or equivalently
\begin{equation*}
f \in V \Longrightarrow T_k f \in V \ \ \forall k \in \Lambda, and \ R_g f \in V \ \  \forall g \in G.
\end{equation*}
Note that $\Gamma$- invariance of a closed subspace is equivalent to that $V\subseteq L^{2}(\textbf{R}) $ is shift invariant, and also $\mathcal{T}V \subseteq L^2(\Omega , l^2(\Lambda^{\perp}))$ is invariant under $\Pi(g)$.

The following proposition, which is \cite[Theorem 3.3]{BCHM}, characterizes $\Gamma$-invariant subspaces of $ L^{2}(\textbf{R})$ in terms of range functions. 
\begin{proposition} \label{mainspace}
A closed subspace $V$ of $L^2 (\textbf{R})$ is $\Gamma$-invariant if and only if it is shift invariant and its range function $J$ satisfies $J(\omega) = \Pi(g) J(\omega)$, where $\Pi(g)$ is the operator defined in \eqref{P}.
\end{proposition}

\section{Main results}
 Now we study properties of operators on $L^2(\textbf{R})$ that commute with  $ T_k R_g $ which  is a unitary representation of $\Gamma = \Lambda \times  G $ on $L^2(\textbf{R})$.  First of all note that for a bounded linear operator $U: L^2(\textbf{R}) \longrightarrow  L^2(\textbf{R})$, using the fiberization map $\mathcal{T}$, one can define an induced  operator $U' :  L^2(\Omega , l^2(\Lambda^{\perp}))  \longrightarrow   L^2(\Omega , l^2(\Lambda^{\perp})) $ given by $U' = \mathcal{T} U \mathcal{T}^{-1}$. In the sequel we define a $\Gamma$-preserving operator on $ L^2(\textbf{R})$ and define an equivalent definition via $U'$. 
\begin{definition}
Let $V$ be a shift invariant and $U: V \longrightarrow  L^2(\textbf{R})$ be a bounded linear operator. Then $U$ is said to be $\Gamma$-preserving if $UT_k R_g = T_k R_g U$, or equivalently if $U$ is shift preserving and $U R_g = R_g U$.
\end{definition}
In the following lemma we prove an equivalent condition for a bounded linear operator on $L^2(\textbf{R})$ to be $\Gamma$-preserving in terms of the induced operator $U'$.
\begin{lemma} \label{pi}
For a shift invariant space $V$, a bounded linear operator $U: V \longrightarrow  L^2(\textbf{R})$ is $\Gamma$-preserving if and only if $U$ is shift preserving and the induced operator $U'$ satisfies $U'\Pi(g) = \Pi(g)U'$.
\end{lemma}
\begin{proof}
We must show that $U R_g = R_g U$ if and only if $U'\Pi(g) = \Pi(g)U'$. If $U R_g = R_g U$, then 
\begin{align*}
\Pi(g)U' &= \mathcal{T} R_g \mathcal{T}^{-1}\mathcal{T} U \mathcal{T}^{-1} \cr
&= \mathcal{T} R_g  U \mathcal{T}^{-1} \cr
&= \mathcal{T} UR_g  \mathcal{T}^{-1} \cr
&= \mathcal{T} U \mathcal{T}^{-1}\mathcal{T} R_g \mathcal{T}^{-1} \cr
&= U'\Pi(g).
\end{align*}
The converse implication has the same proof.
\end{proof}
 In the following theorem, which is the main result of this paper, we give a characterization of $\Gamma$-preserving operator in terms of range operators. Indeed, we show that $\Gamma$-preserving operators on $L^2(\textbf{R})$ are exactly those  shift preserving operators that have an extra condition on their range operators. 
 \begin{theorem} \label{Main}
 Let $V \subseteq L^{2}(\textbf{R})$ be a $\Gamma$-invariant space and  $U: V \longrightarrow  L^2(\textbf{R})$ be a bounded linear operator. Then $U$ is $\Gamma$-preserving if and only if it is shift preserving and also its range operator $\mathcal{R}(\omega)$ satisfies
 \begin{equation} \label{main}
 \mathcal{R}(g^* \omega)= r_{g^{-1}} \mathcal{R}(\omega) r_g,
 \end{equation}
 where $r_g$ is the operators defined by the representation $r$ defined as \eqref{r}.
 \end{theorem}
 \begin{proof}
First assume that $U$ is shift preserving and its range operator satisfies \eqref{main}. Using  Lemma \ref{pi}, it is enough to show that $U'\Pi(g) = \Pi(g)U'$. By \eqref{Prg}, definitions of $U'$ and $\mathcal{R}(\omega)$ we have for $F \in L^2(\Omega , l^2(\Lambda^{\perp}))$
\begin{align*}
\Pi(g)U'F(\omega) & = r_g(U'F(g^* \omega)) \cr 
&= r_g (\mathcal{T} U \mathcal{T}^{-1} F(g^* \omega)) \cr 
&= r_g (\mathcal{R}(g^* \omega)\mathcal{T} \mathcal{T}^{-1} F(g^* \omega)) \cr 
&= r_g (\mathcal{R}(g^* \omega) F(g^* \omega)) \cr 
&= r_g (r_{g^{-1}}\mathcal{R}(\omega)r_g F(g^* \omega)) \cr 
&= \mathcal{R}(\omega)(r_g F(g^* \omega)) \cr
 &= \mathcal{R}( \omega)(\Pi_g F( \omega)) \cr 
&= \mathcal{R}(\omega)(\mathcal{T} R_g \mathcal{T}^{-1} F( \omega)) \cr 
&= \mathcal{T} U R_g \mathcal{T}^{-1} F(\omega) \cr 
&= U'\Pi(g)F(\omega).
\end{align*}
Now suppose that $U$ is $\Gamma$-preserving. Then $U$ is shift preserving, so we must prove that its range operator satisfies \eqref{main}. Put $\mathcal{R}_g (\omega)= r_{g} \mathcal{R}(\omega) r_{g^{-1}}$. Then $\mathcal{R}_g$ is clearly a measurable range operator. On the other hand since $UR_g = R_g U$, by Lemma \ref{pi}, we have $U'\Pi(g) = \Pi(g)U'$, and so for $F \in L^2(\Omega , l^2(\Lambda^{\perp}))$
\begin{align*}
\mathcal{T}U\mathcal{T}^{-1} F(\omega) =  U'F(\omega) &= \Pi(g)U'\Pi(g^{-1})F(\omega) \cr
 &= r_g(U' \Pi(g^{-1})F(g^* \omega)) \cr
&= r_g (\mathcal{R}(g^* \omega)(\Pi(g^{-1})F(g^* \omega))) \cr
&= r_g \mathcal{R}(g^* \omega) ( F((g^{-1})^{*}g^* \omega)) \cr
&= r_g \mathcal{R}(g^* \omega) r_{g^{-1}} F(\omega) \cr
&= \mathcal{R}_g (\omega)(F(\omega)).
\end{align*}
This shows that $\mathcal{R}_g (\omega)$ is a range operator for $U$. Now by uniqueness of the range operator we conclude that $\mathcal{R}_g (\omega)=\mathcal{R}(\omega)$, and so \eqref{main} holds. 
\end{proof}
As an application of  Proposition \ref{mainspace} and Theorem \ref{Main}, we can characterize shift-dilation invariant spaces and shift-dilation preserving operators. Consider  the representation $D: Aut(\textbf{R}) \longrightarrow \mathcal{U}( L^{2}(\textbf{R}))$ given by the dilation operator $D_\alpha f (x) = f(\alpha^{-1}x)$. Note that $Aut(\textbf{R})$ acts on $\textbf{R}$ by $(\alpha, x)\mapsto \alpha(x)$ and the induced action on $\widehat{\textbf{R}}$ is defined by $\langle \alpha^* , \xi \rangle =\langle \xi, \alpha(x) \rangle$, where $\xi \in \widehat{\textbf{R}}$. A closed subspace $V \subseteq L^2 (\textbf{R})$ is said to be shift-dilation invariant if $T_k D_{\alpha} V \subseteq V$, for all $k \in \Lambda$ and $\alpha \in Aut(\textbf{R})$. In this setting the representations \eqref{r} and \eqref{P} are defined as follows
\begin{align*}
r : Aut(\textbf{R}) &\longrightarrow \mathcal{U}(l^2(\Lambda^{\perp})), \ \  r_\alpha a(s)= a(\alpha^* s) \\
\Pi : Aut(\textbf{R}) &\longrightarrow \mathcal{U}( L^2(\Omega , l^2(\Lambda^{\perp}))), \ \ \Pi(\alpha)= \mathcal{T}D_{\alpha} \mathcal{T}^{-1}.
\end{align*}
Now by Proposition \ref{mainspace} we can characterize shift-dilation invariant subspaces as follows.
\begin{corollary}
A closed subspace $V$ of $L^2 (\textbf{R})$ is shift-dilation invariant if and only if it is shift invariant and its range function $J$ satisfies $J(\omega) = \Pi(\alpha) J(\omega)$.
\end{corollary}
Similarly, we can define a shift-dilation preserving operator as a bounded linear operator that commutes with $T_k D_{\alpha}$ for all $k \in \Lambda$ and $\alpha \in Aut(\textbf{R})$. The next corollary, which is an application of Theorem \ref{Main},  characterizes shift-dilation preserving operators on $L^2 (\textbf{R})$.
\begin{corollary}
 Let $V \subseteq L^{2}(\textbf{R})$ be a shift-dilation invariant space and  $U: V \longrightarrow  L^2(\textbf{R})$ be a bounded linear operator. Then $U$ is shift-dilation preserving if and only if it is shift preserving and also its range operator $\mathcal{R}(\omega)$ satisfies
 \begin{equation*}
\mathcal{R}(\alpha^{*} \omega)= r_{\alpha^{-1}} \mathcal{R}(\omega) r_\alpha.
\end{equation*}
\end{corollary}



\bibliographystyle{amsplain}

\end{document}